\documentclass[12pt]{amsart}
\usepackage{doi}
\usepackage{amssymb}

\usepackage{cleveref}
\crefname{lemma}{lemma}{lemmas}
\crefname{clause}{clause}{clauses}
\crefname{claim}{claim}{claims}

\usepackage{enumitem}
\setenumerate{leftmargin=*}
\setitemize{leftmargin=*}

\usepackage{framed}

\theoremstyle{plain}
\newtheorem{theorem}{Theorem}[section]
\newtheorem{theoremIntro}{Theorem}
\newtheorem*{theorem*}{Theorem}

\newtheorem*{conjecture*}{Conjecture}
\newtheorem{lemma}[theorem]{Lemma}

\newtheorem{claim}[theorem]{Claim}

\numberwithin{claim2}{theorem}

\numberwithin{claim3}{claim2}
\newtheorem{corollary}[theorem]{Corollary}
\newtheorem*{corollary*}{Corollary}

\usepackage{etoolbox}

\theoremstyle{definition}
\newtheorem{definition}[theorem]{Definition}

\theoremstyle{remark}

\numberwithin{equation}{theorem}

\newcommand{\PP}{{\mathbb{P}}}

\newcommand{\ga}{\alpha}
\newcommand{\gb}{\beta}
\newcommand{\gc}{\chi}

\newcommand{\gga}{\gamma}

\newcommand{\gj}{\varphi}
\newcommand{\gk}{\kappa}
\newcommand{\gl}{\lambda}
\newcommand{\gm}{\mu}

\newcommand{\gn}{\nu}

\newcommand{\gp}{\pi}

\newcommand{\gr}{\rho}
\newcommand{\gs}{\sigma}
\newcommand{\gt}{\tau}

\newcommand{\gw}{\omega}
\newcommand{\gx}{\xi}

\newcommand{\VN}[1]{\Check{#1}}

\newcommand{\GN}[1]{\Dot{#1}}
\newcommand{\satisfies}{\vDash}
\newcommand{\forces}{\mathrel\Vdash}

\DeclareMathOperator{\Col}{Col}
\newcommand{\union}{\cup}
\newcommand{\bigunion}{\bigcup}
\newcommand{\intersect}{\cap}
\newcommand{\bigintersect}{\bigcap}

\newcommand{\append}{\mathbin{{}^\frown}}

\newcommand{\subelem}{\prec}

\newcommand{\restricted}{\mathord{\restriction}}

\newcommand{\power}[1]{\lvert#1\rvert}
\newcommand{\Pset}{{\mathcal{P}}}
\newcommand{\ordered}[1]{\ensuremath{\langle #1 \rangle}}

\newcommand{\set}[1]{\ensuremath{\{ #1 \}}}
\newcommand{\setof}[2]{\ensuremath{\{ #1 \mid #2 \}}}

\newcommand{\ordof}[2]{\ensuremath{\ordered{ #1 \mid #2 }}}
\newcommand{\formula}[1]{\text{``} #1\text{''}}

\newcommand{\On}{\ensuremath{\text{On}}}

\newcommand{\leftexp}[2]{{\vphantom{#2}}^{#1}{#2}}
\newcommand{\func}{\mathrel{:}}
\DeclareMathOperator{\Chn}{C}	
\DeclareMathOperator{\bChn}{\mathbb{C}}	
\DeclareMathOperator{\dom}{dom}		% domain of function
\DeclareMathOperator{\ran}{ran}
\DeclareMathOperator{\id}{id}

\DeclareMathOperator{\cf}{cf}
\DeclareMathOperator{\tc}{tc}

\DeclareMathOperator{\mo}{o}		% Mitchell order

\DeclareMathOperator{\crit}{crit}	% Critical point of elementary embedding
\DeclareMathOperator{\Ult}{Ult}
\newcommand{\PE}{\mathbb{P}_{E}}

\makeatletter
\g@addto@macro\bfseries{\boldmath}
\makeatother

\title{Some applications of Supercompact Extender Based Forcings to $\HOD$.}
\author{Moti Gitik}
\thanks{The work of the first author was
				 partially supported by ISF grant No.58/14.}
\address{
School of Mathematical Sciences
\\
Raymond and Beverly Sackler Faculty of Exact Sciences
\\
Tel Aviv University
\\
Ramat Aviv 69978
\\
Israel
}
\email{gitik@post.tau.ac.il}
\author{Carmi Merimovich}
\address{
Computer Science School
\\
Tel Aviv Academic College
\\
2 Rabenum Yeroham St.
\\
Tel Aviv
\\
Israel
}
\email{carmi@cs.mta.ac.il}
\date{July 28, 2016}

\newcommand{\OD}{\text{OD}}
\newcommand{\HOD}{\text{HOD}}
\newcommand{\lto}{\mathord{<}}

\begin{document}
\begin{abstract}
 Supercompact extender based forcings are used to construct  models with
   $\HOD$ cardinal structure different from those of $V$.
 In particular, a model with all  regular uncountable cardinals   measurable in
 		 $\HOD$ is constructed.
\end{abstract}
\maketitle
\section{Introduction}
In \cite{CummingsFriedmanMagidorSinapovaRinotPreprint}
	 the following result was proved:
\begin{theorem*}
 Suppose $\kappa < \lambda$ are cardinals such that
 			 $\cf(\kappa) = \omega$,
 			 $\lambda$ is inaccessible,
 			  and $\kappa$ is a limit of
				$\lambda$-supercompact cardinals.
Then there is a forcing poset $Q$ that adds no bounded subsets
	of $\kappa$, and  if $G$ is $Q$-generic then:
\begin{itemize}
\item $\lambda = (\kappa^+)^{V[G]}$.
\item Every cardinal $\geq \lambda$ is preserved in $V[G]$.
\item For every $x\subseteq \kappa$ with $x\in V[G]$,
				$(\kappa^+)^{\HOD_{\set{x}}} < \lambda$.
\end{itemize}
\end{theorem*}
The supercompact extender based Prikry forcing, developed by the second author
	 in \cite{Merimovich2011c}, is applied to reduce largely the
 	 initial assumptions of this theorem and to give a simpler proof.
Namely, we show the following:
\begin{theoremIntro} \label{Thm1}
Suppose $\kappa$ is a $\lto\lambda$-supercompact 
		cardinal\footnote{
			A cardinal $\gk$ is said to be $\lto\gl$-supercompact if there is
			an elementary embedding $j \func V \to M$ such that
				$M$ is transitive,
				$\crit j = \gk$,
		$j(\gk) \geq \gl$, and
			$M \supseteq \leftexp{<\gl}{M}$.
		},
			  and  $\lambda$ is an inaccessible cardinal above $\kappa$.
Then there is a forcing poset $Q$ that
		 adds no bounded subsets of $\kappa$,
		 and if $G$ is $Q$-generic then:
\begin{itemize}
\item 
			$\lambda = (\kappa^+)^{V[G]}$.
\item
			 Every cardinal $\geq \lambda$ is preserved in $V[G]$.
\item
			 For every $x\subseteq \kappa$ with $x\in V[G],
			 		 (\kappa^+)^{\HOD_{\set{x}}} < \lambda$.
\item
			 $\cf^{\HOD_{\set{x}}} \gk = \gw$
\end{itemize}
\end{theoremIntro}
Actually, assuming the measurability (or supercompactness) of $\lambda$ in $V$,
			 we obtain that $(\kappa^+)^{V[G]}$ is measurable (or supercompact)
			 			 in $\HOD_{\set{x}}$.

In \cite{CummingsFriedmanGolshani},
	 a model with the property $(\alpha^+)^{\HOD}<\alpha^+$, for every infinite cardinal $\alpha$ was constructed.
We extend this result, 
		using the supercompact extender based Magidor forcing	 of the second author
				 \cite{MerimovichSupercompactExtender},
		 and show the following:
\begingroup
\makeatletter
\apptocmd{\thetheoremIntro}{\unless\ifx\protect\@unexpandable@protect\protect\footnote{
	This result was presented at the
	Arctic Set Theory Worshop 2 in Kilpisj\"{a}rvi, Finland, February 2015.
	}\fi}{}{}
\makeatother
\begin{theoremIntro}
	 \label{TheBigPicture}
Assume there is a Mitchell increasing sequence of extenders
	$\ordof{E_\gx}{\gx < \gl}$ such that
		$\gl$ is measurable, and
	for each $\gx < \gl$,
		$\crit(j_\gx) = \gk$,
			$M_\gx \supseteq \leftexp{<\gl}{M}_\gx$, and
				$M_\gx \supseteq V_{\gl+2}$, where
					$j_\gx \func V \to \Ult(V, E_\gx) \simeq M_\gx$
					is the natural embedding.
Then there is a model of ZFC where all regular uncountable cardinals are
	measurable in $\HOD$.
\end{theoremIntro}
\endgroup
This may be of some interest due to the following
	 result of H. Woodin \cite{Woodin2010}:
\begin{theorem*}	[The HOD dichotomy theorem]
 \emph{Suppose} $\delta$ is an extendible cardinal.
Then exactly one of the following holds:
\begin{enumerate}
\item
	For every singular cardinal $\gamma>\delta$, $\gamma$ is singular in $\HOD$ 				and $\gamma^+=(\gamma^+)^{\HOD}$
\item
		Every regular cardinal greater than $\delta$ is measurable in $\HOD$.
\end{enumerate}
\end{theorem*}
However, we do not have even inaccessibles in the model of \cref{TheBigPicture}.
It is possible to modify the construction in order to have 
	measurable cardinals (and bit more) in the model.
We do not know how to get supercompacts and
	 it is very
		unlikely the method used  will allow 
				model with supercompacts.

The structure of this work is as follows.
In \cref{sec:HODthings} we give definitions and claims about $\HOD$
	and homogeneous forcing notions which are well know.
In \cref{sec:CofinalityOmega} we prove
	\cref{Thm1}.
In \cref{sec:GlobalResult} we prove \cref{TheBigPicture}.
	
We assume knowledge of large cardinals and forcing.
In particular this work depends
	on the supercompact extender based Prikry-Magidor-Radin forcing.
\section{HOD things} \label[section]{sec:HODthings}
\begin{definition}
Let $M$ be a class.
The class $\OD_M$ contains the sets definable using ordinals and sets from $M$,
i.e., $A \in \OD_M$ iff there is a formula $\gj(x,x_1, \dotsc, x_k, y_1, \dotsc, y_m)$,
	ordinals $\gb, \ga_1, \dotsc, \ga_k \in \On$, and sets $a_1, \dotsc, a_m \in M$, such that
	$A = \setof{a \in V_\gb}{V_\gb \satisfies \gj(a, \ga_1, \dotsc, \ga_k, a_1, \dotsc, a_m)}$.

	The class $\HOD_M$	 contains sets which are hereditarily in $\OD_M$, i.e.,
		$A \in \HOD_M$ iff $\tc(\set{A}) \subseteq \HOD_M$.
		
We write $\OD$ and $\HOD$ for $\OD_\emptyset$ and $\HOD_\emptyset$, respectively.
\end{definition}
Note,  if $A \in \OD$ is a set of ordinals then $A \in \HOD$.

We will work in $\HOD$ of generic extensions, hence
	the relation between $V[G]$ and $\HOD^{V[G]}$,
		where $V[G]$ is a generic extension,
		 will be our main machinery.

Our main tool will be forcing notions which are
	homogeneous in some sense.
	A forcing notion $P$ is said to be cone homogeneous if for each
	pair of conditions $p_0, p_1 \in P$ there is a pair of conditions
			$p_0^*, p_1^* \in P$ such that
					$p_0^* \leq p_0$, $p_1^* \leq p_1$, and
					$P/p_0^* \simeq P/p_1^*$.
					
A forcing notion $P$ is said to be weakly homogeneous if for each
	pair of conditions $p_0, p_1 \in P$ there is an automorphism
		$\gp \func P \to P$ so that $\gp(p_0)$ and $p_1$ are compatible.
It is evident a weakly homogeneous forcing notion is cone homogeneous.

An automorphism $\gp \func P \to P$ induces an automorphism
	on $P$-terms by setting recursively
		$\gp(\ordered{\GN{\gt},p}) = \ordered{\gp(\GN{\gt}), \gp(p)}$.
		
Note ground model terms are fixed by automorphisms,
	i.e., $\gp(\VN{x}) = \VN{x}$,
	in particular for each ordinal $\ga$,
			$\gp(\VN{\ga}) = \VN{\ga}$.

An essential fact about a cone homogeneous forcing notion $P$ is that for
	each formula $\gj$, either
		$\forces_P \gj(\ga_1, \dotsc, \ga_l)$ or
		$\forces_P \lnot\gj(\ga_1, \dotsc, \ga_l)$.
If in addition the forcing $P$ is ordinal definable then
	we get $\HOD^{V[G]} \subseteq V$,
		where $G$ is $P$-generic.

In \cite{DobrinenFriedman2008} it was shown that an arbitrary iteration
	of weakly (cone) homogeneous forcing notions is
		weakly (cone) homogeneous under the  very mild assumption
			that the iterand is fixed by automorphisms.
For the sake of completeness we show here a special case of this
	theorem, which is enough for our purpose.
\begin{theorem}[Special case of
	Dobrinen-Friedman \cite{DobrinenFriedman2008}]
		\label{Dobrinen}
Assume $\ordof{P_\ga, \GN{Q}_\gb}{\ga \leq \gk,\ \gb < \gk}$
	is a backward Easton iteration such that
		for each $\gb < \gk$,
		$\forces_{P_\gb} \formula{ \GN{Q}_\gb \text{ is cone homogeneous}}$
		and for each $p_0, p_1 \in P_\gb$ and automorphism
				$\gp\func P_\gb/p_0 \to P_\gb/p_1$,
					we have
					$\forces_{P_\gb/p_0} \formula{\gp^{-1}(\GN{Q}_\gb) = \GN{Q}_\gb}$.
Then $P_\gk$ is cone homogeneous.
\end{theorem}
\begin{proof}
Fix two conditions $p_0, p_1 \in P_\gk$.
We will construct two conditions $p_0^* \leq p_0$ and $p_1^* \leq p_1$
	such that $P_\gk/p^*_0 \simeq P_\gk/p^*_1$, by which we will be done.
The construction is done by induction on $\ga \leq \gk$ as follows.

Assume $\ga = \gb+1$,
	$p^*_0\restricted \gb$, $p^*_1\restricted \gb$, and
	 $\gp_\gb \func P_\gb / p^*_0\restricted \gb \simeq
	 	 P_\gb/p^*_1\restricted \gb$
	 	 	were constructed.
	We know
	$\forces_{P_\gb/p^*_0\restricted \gb} 
			\formula{\GN{Q}_\gb = \gp_\gb^{-1}(\GN{Q}_\gb)
				\text { is cone homogeneous}}$.
	Let $\gr_\gb \func \GN{Q}_\gb \to \GN{Q}_\gb$
			be a function for which
		$\GN{\gt}[G] = \gr_\gb(\GN{\gt})[\gp_\gb''G]$ holds,
			whenever $G \subseteq P_\gb$ is generic
				and $\GN{\gt}[G] \in \GN{Q}[G]$. 
	If both $p_0(\gb)$ and $p_1(\gb)$ are the maximal element of
		$\GN{Q}_\gb$ then let
			$p^*_0(\gb)$ and $p^*_1(\gb)$ be the maximal element of
					$\GN{Q}_\gb$ and let
						$\gs_\gb = \id$ be the trivial automorphism
							of $\GN{Q}_\gb$.
	If either $p_0(\gb)$ or $p_1(\gb)$ is not the maximal element of
		$\GN{Q}_\gb$ then use
		the  the cone homogeneity of $\GN{Q}_\gb$
	 to find $P_\gb$-names
	$p^*_0(\gb)$,
	$p^*_1(\gb)$, and $\GN{\gs}_\gb$,
	such that $p^*_0 \restricted \gb \forces_{P_\gb}
						 \formula{p^*_0(\gb) \leq p_0(\gb)}$,
		 $p^*_1 \restricted \gb \forces_{P_\gb} \formula{p^*_1(\gb) \leq p_1(\gb)}$,
	 and
		$\GN{\gs}_\gb \func \GN{Q}_\gb/p^*_0(\gb) \simeq
							 \GN{Q}_\gb/\gr_\gb^{-1}(p^*_1(\gb))$
		 is an automorphism.
Whatever way $\GN{\gs}_\gb$ was constructed define
	the automorphism $\gp_{\gb+1}$ by letting
		$\gp_{\gb+1}(s) = \ordered{
						\gp_\gb(s\restricted \gb),
						\gr_\gb(\GN{\gs}_\gb(s(\gb)))
						}$,
	for each			
			 $s \leq p^*_0\restricted \gb+1$.
			
Assume $\ga$ is limit and for each $\gb < \ga$
		we have
			$p^*_0\restricted \gb \leq p_0 \restricted \gb$,
				 $p^*_1\restricted \gb \leq p_1 \restricted \gb$, and
	 $\gp_\gb \func P_\gb / p^*_0\restricted \gb \simeq
	 	 P_\gb/p^*_1\restricted \gb$ is an automorphism
	 	 	such that $\gp_\gb \restricted P_{\gb'} = \gp_{\gb'}$,
	 	 		whenever $\gb' \leq \gb$.
	For each $s \leq p^*_0 \restricted \ga$	let
	$\gp_\ga(s) \in P_\ga$ be the condition defined by setting
		for each $\gb < \ga$,
		$\gp_\ga(s)(\gb) = \gp_{\gb+1}(s\restricted \gb+1)(\gb)$.
\end{proof}
The following claim is practically the successor case of the previous one.
It is useful when we will have automorphism of forcing notions which
	are not necessarily cone homogeneous.
\begin{claim} \label{LiftAutomorphism}
Assume $P_0$ and $P_1$ are forcing notions with
	$\gp_0 \func P_0 \to P_1$ being an isomorphism.
Let $\GN{Q}_0$ be a $P_0$-name of a cone homogeneous forcing notion
	 such that
	$\forces_{P_0} \formula{\GN{Q}_0 = \GN{Q}_1}$, where
		$\GN{Q}_1 = \gp_0(\GN{Q}_0)$.

Then for each pair $1*\GN{q}_0 \in P_0 * \GN{Q}_0$ and
			$1*\GN{q}_1 \in P_1 * \GN{Q}_1$
			there are stronger conditions
				$1*\GN{q}^*_0 \leq 1*{\GN{q}_0}$ and
			$1*\GN{q}^*_1 \leq 1*{\GN{q}_1}$ such that
				$P_0 * \GN{Q}_0/1*\GN{q^*_0} \simeq
								P_1 * \GN{Q}_1/1*\GN{q}_1^*$.
\end{claim}
\begin{proof}
Note there is a function $\gr$ taking $P_0$-names to $P_1$-names
	such that $\GN{q}_0[G_0] = \gr(\GN{q}_0)[G_1]$,
		where $G_0 \subseteq P_0$ is generic and $G_1 = \gp_0''G_0$.

Set $\GN{q}'_1 = \gr^{-1}(\GN{q}_1)$.
By the cone homogeneity of $\GN{Q}_0$ in $V^{P_0}$ there are
		 stronger conditions
	$\GN{q}^*_0 \leq \GN{q}_0$ and
	$\GN{q}^{\prime*}_1 \leq \GN{q}'_1$,
		for which there is  (a name of) an automorphism
		$\gp_1 \func \GN{Q}_0/\GN{q}^*_0 \to \GN{Q}_0/\GN{q}_1^{\prime*}$.
		Set $\GN{q}_1^* = \gr(\GN{q})_1^{\prime*}$.
		Since for generics $G_0, G_1$ as above we
	have $\GN{Q}_0/\GN{q}_1^{\prime*}[G_0] =
			\GN{Q}_1/\GN{q}_1^{*}[G_1]$ we get
				 $\gp(p*\GN{q}) = 
							\gp_0(p) * (\gr \circ \gp_1(\GN{q}))$ is the required automorphism.
\end{proof}
		
While the forcing notions we will use are cone homogeneous
	we will deliberately break some of their homogeneity.
	The relation between $\HOD^{V[G]}$ and $V$
		will be as follows.
\begin{claim} \label{HodIsInSmallExtension}
Assume $P$ is an ordinal definable cone homogeneous forcing notion.
Let $\gp \func P \to P$ be a projection.
Assume for each condition $p\in P$ and ordinals $\ga_1,\dotsc,\ga_l \in \On$,
	if $p \forces_P \gj(\ga_1, \dotsc, \ga_l)$ then
			$\gp(p) \forces_P \gj(\ga_1, \dotsc, \ga_l)$.
Then $\HOD^{V[G]} \subseteq V[\gp''G]$.
\end{claim}
\begin{proof}
Assume $\forces_P \formula{\GN{A} \subseteq \On
					 \text{ and }\GN{A}\in \HOD}$.
Let $G \subseteq P$ be generic.
Then in $V[G]$ there are ordinals $\ga_1, \dotsc, \ga_l,\gb$ such that
	for each $\ga \in \On$,
\begin{align*}
	\ga \in \GN{A}[G] \iff V_\gb \satisfies \gj(\ga, \ga_1, \dotsc, \ga_l).
\end{align*}
Let $X^\ga_0 \union X^\ga_1 \subseteq P$ be
	a maximal antichain such that for each $p \in X^\ga_0$,
	\begin{align*}
		p \forces V_\gb \satisfies \lnot \gj(\ga,\ga_1, \dotsc, \ga_l),
	\end{align*}
	and for each $p \in X^\ga_1$,
	\begin{align*}
		p \forces V_\gb \satisfies \gj(\ga,\ga_1, \dotsc, \ga_l).
	\end{align*}
	Let $\GN{A}'$ be a $\gp''P$-name defined by setting
			 for each $p \in X^\ga_0 \union X^\ga_1$.
	\begin{align*}
		\gp(p) \forces_{\gp''P} \formula{\ga \in \GN{A}'} \iff
										 p \forces_P \formula{\ga \in \GN{A}}.
	\end{align*}
	Since $\gp''(X_0 \union X_1)$ is predense in $\gp''P$ we get
		$\GN{A}'[\gp''G] = \GN{A}[G]$,
		by which we are done.
\end{proof}

Let $C(\gt, \gm)$ be the Cohen forcing  for adding $\gm$ subsets to $\gt$, i.e.,
	$C(\gt, \gm) = \setof {f \func a \to 2}{a \subseteq \gm, \ \power{a}<\gt}$.
The following is well known.
\begin{claim} \label{CohenHomogen}
$\Chn(\gt, \gm)$ is cone homogeneous.
\end{claim}
\begin{proof}
Assume $f, g \in \bChn(\gt,\gm)$ are conditions.
Choose stronger conditions, $f^* \leq f$ and $g^* \leq g$, such that
	$\dom f^* = \dom g^* = \dom f \union \dom g$.
Define $\gp \func \bChn(\gt, \gm)/f^* \to \bChn(\gt, \gm)/g^*$ by setting
	$\gp(f') = g^* \union (f' \setminus f^*)$
		for each $f' \leq f^*$.
It is obvious $\gp$ is an automorphism.
\end{proof}
The following is immediate from the previous claim and \cref{Dobrinen}.
\begin{claim}
The Easton product of Cohen forcing notions is cone homogeneous.
\end{claim}
\section{The cofinality $\gw$ case} \label{sec:CofinalityOmega}
Let us switch to the cone-homogeneity of the Extender Based Prikry forcing
	(\cite{GitikMagidor1992}).
Let $E$ be an extender as in \cite{Merimovich2011c} or \cite{MerimovichSupercompactExtender}.
Let $\PE$ be the extender based Prikry forcing derived from $E$.
We show $\PE$ is cone homogeneous.
\begin{claim} \label{HomogeneousExtenderBased} \label{MainHomogenTool}
For a pair of conditions $p_0, p_1 \in \PE$ there are direct extensions
		 $p_0^* \leq ^* p_0$  and $p_1^* \leq^* p_1$
	such that $\PE/p_0^* \simeq \PE/p_1^*$.
\end{claim}
\begin{proof}
Set $d = \dom f^{p_0} \union \dom f^{p_1}$.
Set $f_0^* = f^{p_0} \union \setof {\ordered{\ga, \ordered{}}}
										{\ga \in d \setminus \dom f^{p_0}}$ and
		$f_1^* = f^{p_1} \union \setof {\ordered{\ga, \ordered{}}}
											{\ga \in d \setminus \dom f^{p_1}}$.
Choose a set $A \subseteq
					 \gp^{-1}_{d, \dom f^{p_0}}(A^{p_0}) \intersect
					 \gp^{-1}_{d, \dom f^{p_1}}(A^{p_1})$
	 so that
	both $p_0^* = \ordered{f_0^*, A}$ and $p_1^* = \ordered{f_1^*, A}$
					are conditions.
Define $\gp \func \PE/p_0^* \to \PE/p_1^*$
	by setting for each $p \leq p_0^*$,
		$\gp(p) = \ordered{
				f^{p_1^*}_{\ordered{\gn_0, \dotsc, \gn_{n-1}}} \union
										(f^{p} \restricted (\dom f^{p} \setminus d)),
				A^{p}
			}$, where
	$\ordered{\gn_0, \dotsc, \gn_{n-1}} \in \leftexp{<\gw}{A^{p}}$ and
			$p \leq^* p^*_{0\ordered{\gn_0, \dotsc, \gn_{n-1}}}$.
It is evident $\gp$ is an automorphism of $\PE$.
\end{proof}
For a generic filter $G \subseteq \PE$ define the function $f_G$
	by setting $f_G(\ga) = \bigunion \setof {f^p(\ga)}{p \in G, \ga \in \dom f^p}$.
	
Let us define the Easton products we are going to work with.
Let $A \subseteq \On$ be a set of ordinals.
Let $\bChn_{\gc,A}$ be the Easton product of the Cohen forcing notions
		yielding, in the generic extension, for each $\gx < \sup A$,
		\begin{align*}
			2^{\gc^{+\gx+1}} = \begin{cases}
				\gc^{+\gx+3}	&	\gx \in A, \\
				\gc^{+\gx+2}	& \gx \notin A.
			\end{cases}
		\end{align*}
When forcing with  $\bChn_{\gc,A}$
	 we will choose $\gc$ to be large enough
	so as not to interfere with our intended usage.
	Due to the (cone) homogeneity of $\PE$, the seuqences forced by $\PE$
	are not in $\HOD^{V^{\PE}}$.
We would like to break the homogeneity of $\PE$ so as to have the
	Prikry sequence enter $\HOD^{V^{\PE}}$.
We will achieve this  by coding the Prikry sequence into the power set function.
We will want the Cohen forcing used to be stabilized by
	reasonable automorphisms of $\PE$.
Thus define the projection $s \func \PE \to \PE$ by setting
	$s(p) = \ordered{f^p\restricted \set{\gk}, A^p\restricted \set{\gk}}$,
		where $A^p\restricted \set{\gk} =
			\setof {\gn\restricted \set{\gk}}{\gn \in A^p}$.
Note $s''\PE$ is the usual Prikry forcing based on $E(\gk)$.
Moreover if $G \subseteq \PE$ is generic then
	$s''G$ is $s''\PE$-generic.
\begin{claim} \label{ProductIso}
Let $\PP = \PE * \GN{\bChn}_{\gc,\GN{f}_G(\gk)}$.
Assume $\ordered{p_0, \GN{q}_0}, \ordered{p_1,\GN{q}_1} \in \PP$
	are conditions
	such that $s(p_0)$ and $s(p_1)$
		 are compatible.
Then there are stronger conditions,
		$\ordered{p^*_0, \GN{q}^*_0} \leq \ordered{p_0, \GN{q}_0}$ and
			$\ordered{p^*_1, \GN{q}^*_1} \leq \ordered{p_1, \GN{q}_1}$,
	such that $\PP/\ordered{p^*_0, \GN{q}^*_0} \simeq
							\PP/\ordered{p^*_1, \GN{q}^*_1}$.
\end{claim}
\begin{proof}
Since		$s(p_0)$ and $s(p_1)$ are compatible,
	 we can
	choose conditions $p'_0 \leq p_0$ and $p'_1 \leq p_1$ such that
	$f^{p'_0}\restricted \set{\gk} = f^{p'_1}\restricted \set{\gk}$.
By \cref{HomogeneousExtenderBased}	there are direct extensions $p^*_0 \leq^* p'_0$ and $p^*_1 \leq^* p'_1$
	such that $\gp_0 \func \PE/p_0^* \simeq \PE/p_1^*$ is an automorphism.
Since $\bChn_{\gc,f_{G}(\gk)} = \gp(\bChn_{\gc,f_{G}(\gk)})$,
	where $G \subseteq \PE$ is generic,
	we are done by \cref{LiftAutomorphism}.
\end{proof}
The following is immediate from the previous claim.
\begin{corollary}
Assume $\ga, \ga_1, \dotsc, \ga_n \in \On$ and
	$\ordered{p,q} \forces_{\PP} \gj(\ga,\ga_1, \dotsc, \ga_n)$.
Then $\ordered{s(p),1} \forces_{\PP}
		 \gj(\ga, \ga_1, \dotsc, \ga_n)$.
\end{corollary}
\begin{proof}
	In order to show
	$\ordered{s(p), 1} \forces_{\PP}
		 \gj(\ga, \ga_1, \dotsc, \ga_n)$ we will show
		 	a dense subset of conditions below $\ordered{s(p), 1}$
		 			forces $\gj(\ga, \ga_1, \dotsc, \ga_n)$.
Let $\ordered{p_0, \GN{q}_0} \leq \ordered{s(p), 1}$ be
	an arbitrary condition.
By \cref{ProductIso} there is
	$\ordered{p_0', \GN{q}_0'} \leq \ordered{p_0, \GN{q}_0}$
			and $\ordered{p_1', \GN{q}_1'} \leq \ordered{p, \GN{q}}$ such that
			 $\PP/p_0'*\GN{q}_0' \simeq \PP/p_1'*\GN{q}_1'$.
	Thus $\ordered{p'_0,\GN{q}'_0} \forces_\PP
				\gj(\ga_1, \dotsc, \ga_n)$.
\end{proof}
The previous corollary together with
		 \cref{HodIsInSmallExtension} yields the following.
\begin{corollary} \label{AlmostThere0}
Assume $G*H$ is $\PP$-generic.
Then $\cf^{V[G*H]} \gk = \gw$ and
		$f_G(\gk) \in \HOD^{V[G*H]} \subseteq V[s''G]$.
\end{corollary}
We will get a special case of \cref{Thm1} by invoking the last corollary
	in a model of the form $L[A]$.
\begin{corollary} \label{Thm}
Assume $V = L[A]$, where $A \subseteq \On$ is a set of ordinals,
	and $E$ is an extender witnessing $\gk$ is a $\lto\gl$-supercompact cardinal.
	There is a forcing notion $R$ preserving the extender $E$
	such that in $V[I][G*H]$, where $I*G*H$ is $R*\PP$-generic,
	$\gk^+ = \gl$, $\cf\gk = \gw$,
	and
		$\HOD^{V[I][G][H]} = V[I][s''G]$.
\end{corollary}
\begin{proof}
We will begin by defining the forcing notion $R$ so that
	for an $R$-generic filter $I$  we will have $\HOD^{V[I]} = V[I]$.
	
Define by induction the forcing notions $\ordof {R_n}{n\leq\gw}$ and sets $\ordof{A_n}{n<\gw}$,
	as follows.
Set $R_0 = {1}$ and $A_0 = A$.
For each $n<\gw$ define $R_{n+1}$ as follows.
In $V[G_n]$,
	where $G_n \subseteq R_n$ is generic over $V$,
		 let $\bChn_n$ be the forcing notion $\bChn_{\gc_n, A_n}$.
	Let $A_{n+1}$ be $\bChn_{n}$-generic over $V[G_n]$,
		 i.e., $A_{n+1}$ is a code for $A_n$.
Set $R_{n+1} = R_n * \GN{\bChn}_{n}$, where
	$\GN{\bChn}_n$ is an $R_n$-name for $\bChn_n$.
Let $R$ be the inverse limit of $\ordof {R_n}{n<\gw}$.
Let $I\subseteq R$ be generic.

Invoking $\cref{AlmostThere0}$ inside $V[I]$ and calculating
	$\HOD^{V[I][G][H]}$ we get
		$f_G(\gk) \in \HOD^{V[I][G][H]} \subseteq V[I][s''G]$.
	For each $n<\gw$, $A_n \in \HOD^{V[I][G][H]}$,
			thus $\HOD^{V[I][G][H]} \supseteq L[A][I][s''G]=
							V[I][s''G]$.
\end{proof}
			Hence we get:
\begin{corollary}
Assume $\gl$ is measurable and $\gk$ is $\lto\gl$-supercompact.
Then there is a generic extension in which
	 $\cf^{\HOD} \gk = \gw$, and
	$\gk^+$ (of the generic extension) is
		$\HOD$-measurable.
\end{corollary}
In order to analyze $\HOD_{\set{a}}$, where $a \subseteq \gk$,
	 let us derive another line of corollaries 	stemming from \cref{ProductIso}.
The problem we face when dealing with $\HOD_{\set{a}}$ is
	an automorphism $\gp$ of $\PP$ might move $\GN{a}$, the name of $a$.
Thus we will need to fine tune the projection $s$.

First we recall the notion of good pair from
 \cite{MerimovichSupercompactExtender}.
We say the pair $\ordered{N, f}$ is a good pair
	if $N \subelem H_\gc$ is a $\gk$-internally approachable elementary substructure
		and there is a sequence
					 $\ordof{\ordered{N_\gx, f_\gx}}{\gx<\gk}$ such that
				$\ordof {N_\gx}{\gx < \gk}$ witnesses the $\gk$-internal approachablity
						of $N$,
							$f = \bigunion \setof {f_\gx}{\gx < \gk}$,
							 $\ordof{f_\gx}{\gx<\gk}$ is a $\leq^*$-decreasing
								continuous sequence in $\PP^*_f$,
									 and for each $\gx < \gk$,
							$f_{\gx} \in \bigintersect
									\setof {D \in N_\gx}{D \text{ is a dense open subset of }
																\PP^*_f}$,
												$f_\gx \subseteq N_{\gx+1}$,
											 and
												$f_\gx \in N_{\gx+1}$.

Define the projection $s_N \func \PE\to \PE$ by setting
	for each $p \in \PE$,
	$s_N(p) = \ordered{f^p\restricted N, A^p\restricted N}$.
\begin{corollary}
Assume $N \subelem H_\gc$ is an elementary substructure such that
	$p^*$ is an $\ordered{N, \PE}$-generic condition and
			$\ordered{N, f^{p^*}}$ is a good pair.
Let $\GN{a} \in N$ be a $\PE$-name such that
		$\forces_{\PE} \formula{\GN{a} \subseteq \gk}$.			
If $\ga, \ga_1, \dotsc, \ga_n \in \On$,
	$p \leq p^*$,
	and
	$\ordered{p,\GN{q}} \forces_{\PP} \gj(\ga,\ga_1, \dotsc, \ga_n, \GN{a})$,
	then
		 $\ordered{s_N(p), 1} \forces_{\PP}
		 \gj(\ga, \ga_1, \dotsc, \ga_n, \GN{a} )$.
\end{corollary}
\begin{proof}
	In order to show
	$\ordered{s_N(p), 1} \forces_{\PP}
		 \gj(\ga, \ga_1, \dotsc, \ga_n, \GN{a} )$ we will show
		 	a dense subset of conditions below $\ordered{s_N(p), 1}$
		 			forces $\gj(\ga, \ga_1, \dotsc, \ga_n, \GN{a} )$.

Let $\ordered{p_0, \GN{q}_0} \leq \ordered{s_N(p), 1}$ be
	arbitrary condition.
We can choose $p_1 \leq p$ such that
	$s_N(p_0) = s_N(p_1)$.
By \cref{HomogeneousExtenderBased} there is $p_0^* \leq^* p_0$ and
			$p_1^* \leq^* p_1$ such that
				$\PE/p^*_0 \simeq \PE/p^*_1$.
				
Recall that if $r \leq p^*$, $\ga < \gk$,
	 and $r \forces_{\PE} \formula{\ga \in \GN{a}}$,
	then $p^*_{\ordered{\gn_0,\dotsc, \gn_{l-1}}} \forces
							 \formula{\ga \in \GN{a}}$,
		where 
				$\ordered{\gn_0, \dotsc, \gn_{l-1}} \in \leftexp{<\gw}{A}^{p^*}$
					is such that
					$r \leq^* p^*_{\ordered{\gn_0,\dotsc, \gn_{l-1}}}$.
Thus for each $\ordered{\gn_0, \dotsc, \gn_{l-1}} \in A^{p^*_0} = A^{p^*_1}$,
	$\ga < \gk$, and $r \in \PE/p^*_0$,
\begin{align*}
	&r \leq^* p^*_{0\ordered{\gn_0, \dotsc, \gn_{l-1}}}
		\text{ and }
		r \forces_{\PE} \formula{\ga \in \GN{a}} \iff \\
		& p_{\ordered{\gn_0, \dotsc, \gn_{l-1}}\restricted \dom f^{p}}
				\forces_{\PE} \formula{\ga \in \GN{a}} \iff \\
		& \gp(r) \leq^* p^*_{1\ordered{\gn_0, \dotsc, \gn_{l-1}}}
		\text{ and }
						\gp(r) \forces_{\PE} \formula{\ga \in \gp(\GN{a})}		.
\end{align*}				
Thus $p^*_0 \forces \formula{\GN{a} = \gp^{-1}(\GN{a})}$.
Use	\cref{ProductIso} to find stronger conditions
	$\ordered{p_0', \GN{q}_0'} \leq \ordered{p^*_0, \GN{q}_0}$
			and $\ordered{p_1', \GN{q}_1'} \leq \ordered{p^*_0, \GN{q}}$
			 such that
			 $\Tilde{\gp}\func \PP/p_0'*\GN{q}_0' \simeq \PP/p_1'*\GN{q}_1'$
			 	is an automorphism.
	Since
		 $\ordered{p'_1,\GN{q}'_1} \forces_{\PP}
				\gj(\ga_1, \dotsc, \ga_n, \GN{a})$
	we get
		$\ordered{p'_0,\GN{q}'_0} \forces_{\PP}
				\gj(\ga_1, \dotsc, \ga_n, \gp^{-1}(\GN{a}))$.
		We are done since 		
			$p'_0 \forces \formula{\GN{a} = \gp^{-1}(\GN{a})}$.
\end{proof}
\begin{corollary} \label{AlmostThere1}
Assume $G*H$ is $\PP$-generic, $a \in V[G*H]$, and $a \subseteq \gk$.
Then $\cf^{V[G*H]} \gk = \gw$
	and
		$f_G(\gk) \in \HOD^{V[G*H]}_{\set{a}} \subseteq V[s_X''G]$
	for a set $X \subseteq \dom E$  such that $\power{X} < \gl$.
\end{corollary}
\newcommand{\sA}{\mathfrak{A}}
We will get \cref{TheBigPicture} by beginning with a model where
	$\HOD \supseteq V_{\gl+2}$.
For this let us define the following coding.
Let $\sA = \ordof{A_\ga}{\ga < \gl^{+3}}$ be an enumeration of
	all subsets of $\gl^{++}$.
Let $\bChn_{\gc, \sA}$ be the Easton product of the Cohen forcing notions
	yielding, in the generic extension, for each $\ga < \gl^{+3}$ and
			$\gx < \gl^{++}$,
					\begin{align*}
			2^{\gc^{+\gl^{++}\cdot\ga+\gx+1}} = \begin{cases}
				\gc^{\gl^{++}\cdot\ga+\gx+3}	&	\gx \in A_\ga, \\
				\gc^{\gl^{++}\cdot\ga+\gx+2}	& \gx \notin A_\ga.
			\end{cases}
		\end{align*}

\begin{corollary}
Let
	$E$ is an extender witnessing $\gk$ is a $\lto\gl$-supercompact cardinal.
	In  $V[I][G*H]$, where $I*G*H$ is $\Chn_{\gc,\sA}*\PP$-generic,
	$\gk^+ = \gl$, 
	and for each set $a \subseteq \gk$,
	$\cf^{\HOD_{\set{a}}^{V[I][G*H]}}\gk=\gw$	and
		$\gl$	 is $\HOD_{\set{a}}^{V[I][G][H]}$-measurable.
\end{corollary}
\begin{proof}
Let $U \in V$ be a measure on $\gl$.
Then $U \in V_{\gl+2}$, hence
	 $U \in \HOD^{V[I]}$, where $I$ is $\Chn_{\gc, \sA}$-generic.

Working in $V[I]$ let $G*H$ be $\PP$-generic.
By \cref{AlmostThere1}  there is $X \subseteq \dom E$ such that
	$\power{X}< \gl$, $X \in V[I]$, and
		$f_G(\gk) \in \HOD_{\set{a}}^{V[I][G*H]} \subseteq V[I][s_X''G]$.	
The filter $s_X''G$ is $s_X''\PE$-generic.
Since $\power{X} < \gl$ we have $\power{s_X''\PE}<\gl$,
	hence any $V$-measure over $\gl$
		trivially lifts to a $V[s_X''G]$-measure over $\gl$.
	In particular $U$ lifts to $\Bar{\Bar{U}}$, which is
		definable by
		$\Bar{\Bar{U}} = \setof {B \in V[I][s_X''G] \intersect \Pset(\gl)}
				{\exists A\in U\, B \supseteq A}$.
Since $U  \in\HOD_{\set{a}}^{V[I][G*H]}$ we can define in
	$\HOD_{\set{a}}^{V[I][G*H]}$,
	$\Bar{U} =
			\setof {B \in \HOD_{\set{a}}^{V[I][G*H]} \intersect \Pset(\gl)}
				{\exists A\in U\, B \supseteq A}$.
Since $\HOD_{\set{a}}^{V[I][G*H]} \subseteq V[I][s_X''G]$ we necessarily
	have $\Bar{U} \subseteq \Bar{\Bar{U}}$.
	Thus
		$\Bar{U}$ is a measure on $\gl$ in $\HOD_{\set{a}}^{V[I][G*H]}$.
\end{proof}
\renewcommand{\PE}{\PP_{\Vec{E}}}
\section{The global result} \label{sec:GlobalResult}
In this section we  prove \cref{TheBigPicture}.
Thus throughout this section assume
	$\Vec{E}= \ordof{E_\gx}{\gx < \gl}$ is a Mitchell increasing sequence
 of extenders
	such that
		$\gl$ is measurable, and
	for each $\gx < \gl$,
		$\crit(j_\gx) = \gk$,
			$M_\gx \supseteq \leftexp{<\gl}{M}_\gx$, and
				$M_\gx \supseteq V_{\gl+2}$, where
					$j_\gx \func V \to \Ult(V, E_\gx) \simeq M_\gx$
							is the natural embedding.
(We demand $M_\gx \supseteq V_{\gl+2}$ since we want
	$\gl$ to be measurable in all ultrapowers, not only in $V$).

Let $\PE$ be the supercompact extender based Radin forcing
	using $\Vec{E}$. (see
			 \cite{Merimovich2011b,MerimovichSupercompactExtender}).
Let us deal with the homogeneity of the Extender Based Radin forcing
\begin{lemma} \label{HomogeneousExtenderBasedRadin0}
 \label{MainHomogenToolRadin}
For a pair of conditions $p_0, p_1 \in \PE^*$ there are direct extensions
		 $p_0^* \leq ^* p_0$  and $p_1^* \leq^* p_1$
	such that $\PE/p_0^* \simeq \PE/p_1^*$.
\end{lemma}
\begin{proof}
Set $d = \dom f^{p_0} \union \dom f^{p_1}$.
Set $f_0^* = f^{p_0} \union \setof {\ordered{\ga, \ordered{}}}
										{\ga \in d \setminus \dom f^{p_0}}$ and
		$f_1^* = f^{p_1} \union \setof {\ordered{\ga, \ordered{}}}
											{\ga \in d \setminus \dom f^{p_1}}$.
Choose a set $T \subseteq
					 \bigunion_{\gx<\mo(\Vec{E})}\gp^{-1}_{\gx, d, \dom f^{p_0}}
					 			(T^{p_0}) \intersect
					 \bigunion_{\gx<\mo(\Vec{E})}\gp^{-1}_{\gx, d, \dom f^{p_1}}
					 			(T^{p_1})$
	 so that
	both $p_0^* = \ordered{f_0^*, T}$ and $p_1^* = \ordered{f_1^*, T}$
					are conditions.
Define $\gp \func \PE/p_0^* \to \PE/p_1^*$
	by setting $\gp(p^0)$  for each $p^0 \leq p_0^*$ as follows.
Let $\ordered{\gn_0, \dotsc, \gn_{n-1}} \in \leftexp{<\gw}{T}^{p_0^*}$
	such that $p^0 \leq^* p^*_{0\ordered{\gn_0, \dotsc, \gn_{n-1}}}$.
	 Let $p^0 = p^0_0 \append \dotsb p^0_n$ and
	 $p^*_{1\ordered{\gn_0, \dotsc, \gn_{n-1}}} =
	 		 p^{1*}_0 \append \dotsb \append p^{1*}_n$.
	 Let $\gp(p^0) =  p^1_0 \append \dotsb \append p^1_n$,
	 	where
	 			$p^1_i = \ordered{f^{p^{1*}_i} \union f^{p^0_i}\restricted
	 				(\dom f^{p^0_i} \setminus \dom f^{p^{1*}_i}),
	 				T^{p^0_i}}$.
It is evident $\gp$ is an automorphism.
\end{proof}
Recall that for a condition $p = p_0 \append \dotsb \append p_n$
	we have $\PE/p \simeq \PP_{\Vec{e}_0}/p_0 \dotsb \append \PP_{\Vec{e}_n}/p_n$,
	where $p_i \in \PP^*_{\Vec{e}_i}$ and $\Vec{e}_n = \Vec{E}$.
Thus the following is an immediate corollary of the above lemma by recursion.
\begin{corollary} \label{HomogeneousExtenderBasedRadin}
Assume $p^0,p^1 \in \PE$ are conditions such that
	$p^0,p^1 \in \prod_{0 \leq i \leq n} P^*_{\Vec{e}_i}$.
Then there are direct extensions $p^{0*} \leq^* p^0$ and
	$p^{1*} \leq^* p^1$ such that $\PE/p^{0*} \simeq \PE/p^{1*}$.
\end{corollary}
For a condition $p \in \PE^*$ define its projection $s(p)$
 to the normal measure 	 by setting
		$s(p) = \ordered{f^p\restricted \set{\gk}, T^p\restricted \set{\gk}}$.
Define by recursion the projection of arbitrary condition
	$p = p_0 \append \dotsb \append p_n \in \PE$ by setting
		$s(p) = s(p_0 \append \dotsb \append p_{n-1}) \append s(p_n)$.
It is obvious $s''\PE$ is the Radin forcing using the measures
	$\ordof{E_\gx(\gk)}{\gx < \mo(\Vec{E})}$.
Moreover, if $G$ is $\PE$-generic then $s''G$ is $s''\PE$-generic.

Let $G$ be $\PE$-generic.
Work in $V[G]$.
Let $\ordof {\gk_\ga}{\ga < \gk}$ be the increasing enumeration of
	$f_G(\gk)$.
Define the sequence $\ordof{\gm_\ga, U_\ga}{\ga<\gk}$ by setting
	for each $\ga < \gk$,
	\begin{align*}
		\gm_\ga = \begin{cases}
			\gk_\ga^+	&	\ga \text{ is limit}, \\
			\gk_\ga			& \ga \text{ is successor}.
		\end{cases}
	\end{align*}
Note: If $\ga$ is limit, then $\gm_\ga = \gk_\ga^+$ is $V$-measurable
	since it is a reflection of $\gl$ being measurable in one of
			the $V$-ultrapowers.
On the other hand,
	if $\ga$ is successor then $\gm_\ga = \gk_\ga$ is $V$-measurable	 since
	$E_0$ concentrates on measurables.
Thus for each $\ga < \gk$ we can choose
	$U_\ga \in V$ which is a $V$-measure over $\gm_\ga$.	
Define the backward Easton iteration
	$\ordof {P_\ga, \GN{Q}_\gb}{\ga  \leq \gk,\, \gb< \gk}$
		by setting for each $\ga < \gk$,
	$\GN{Q}_\ga = \Col(\gm_\ga, \lto\gk_{\ga+1})$.
By \cref{Dobrinen} the iteration $P_\gk$ is cone homogeneous.
Let $H \subseteq P_\gk$ be generic.

Working in $V[G*H]$ we want to pull into the $\HOD$ of a generic extension
	the $V$-measures $U_\ga$'s.
Define the backward Easton iteration
	$\ordof{R_\ga, \GN{S}_\gb}{\ga\leq \gk, \ \gb < \gk}$
		by setting for each $\gb < \gk$,
			 $\GN{S}_\gb = \Chn_{\gc_\gb,\mathfrak{A}_\gb}$, where,
				$\mathfrak{A}_\gb = \setof {A \in V}
						{A \subseteq (\gm_\gb^{++})_V}$ and
					$\sup_{\gga <\gb} \gc_\gga < \gc_\gb < \gk$.
By \cref{Dobrinen} $R_\gk$ is cone homogeneous.

One final definition is in order before the following claim.
If $p \in \PE^*$ then set $\gk(p) = \ran f^p(\gk)$.
If $p = p_0 \append \dotsb \append p_n \in \PE$ then
	set by recursion
		$\gk(p) = \gk(p_0 \append \dotsb \append p_{n-1}) \append \gk(p_n)$.
Note $\gk(p)$ is the subset of $f_G(\gk)$ decided by the condition $p$.
\begin{claim} \label{ProductIsoRadin}
Let $\PP = \PE * \GN{P}_\gk * \GN{R}_\gk$.
Assume $\ordered{p_0, \GN{q}_0, \GN{q}_0},
		 \ordered{p_1,\GN{q}_1, \GN{r}_1} \in \PP$
	are conditions
	such that $s(p_0)$ and $s(p_1)$
		 are compatible.
Then there are stronger conditions,
		$\ordered{p^*_0, \GN{q}^*_0, \GN{r}^*_0} \leq
				\ordered{p_0, \GN{q}_0, \GN{r}_0}$ and
			$\ordered{p^*_1, \GN{q}^*_1, \GN{q}^*_1} \leq
						 \ordered{p_1, \GN{q}_1, \GN{r}_1}$,
	such that $\PP/\ordered{p^*_0, \GN{q}^*_0, \GN{r}^*_0} \simeq
							\PP/\ordered{p^*_1, \GN{q}^*_1, \GN{r}^*_1}$.
\end{claim}
\begin{proof}
Since		$s(p_0)$ and $s(p_1)$ are compatible
	there are stronger conditions $p'_0 \leq p_0$ and $p'_1 \leq p_1$
		and 	Mithcell increasing sequences $\setof{\Vec{e}_i}{i \leq k}$
	such that
				 $p'_0, p'_1 \in
			 			\prod_{i\leq k} \PP_{\Vec{e_i}}$ and
			 			$\gk(p'_0) = \gk(p'_1)$.
	By the previous corollary there are direct extensions
		$p_0^* \leq^* p_0'$ and $p_1^{*} \leq^* p'_1$
				 	such that
				 		$\gp\func \PE/p_0^{*} \simeq \PE/p_1^{*}$.
	Most importantly we have
			$\gp(\GN{P}_\gk*\GN{Q}_\gk) = \GN{P}_\gk*\GN{Q}_\gk$
				is cone homogeneous.
	Thus by \cref{LiftAutomorphism} we are done.
\end{proof}
\begin{corollary}
If $\ordered{p, \GN{q}, \GN{r}} \forces_{\PP}
				\gj(\ga_1,\dotsc,\ga_l)$,
	then $\ordered{s(p), 1, 1} \forces_{\PP} \gj(\ga_1,\dotsc,\ga_l)$.
\end{corollary}
\begin{proof}
We will prove a dense subset of conditions below $\ordered{s(p), 1, 1}$
	force $\gj(\ga_0, \dotsc, \ga_{l})$.
Assume $\ordered{p^0, \GN{q}^0, \GN{r}^0} \leq \ordered{s(p), 1, 1}$.
Trivially $s(p^0)$ and $s(p)$ are compatible, hence
	by the previous corollary there are
		stronger conditions
			$\ordered{p^{0*}, \GN{q}^{0*}, \GN{r}^{0*}}
					\leq \ordered{p^0, \GN{q}^0, \GN{r}^0}$ and
			$\ordered{p^{1*}, \GN{q}^{1*}, \GN{r}^{1*}}
			\leq
				\ordered{p, \GN{q}, \GN{r}}$ such that
				$\PP/\ordered{p^{0*}, \GN{q}^{0*}, \GN{r}^{0*}}
				 \simeq \PP/\ordered{p^{1*}, \GN{q}^{1*}, \GN{r}^{1*}}$.
			Necessarily	$\ordered{p^{0*}, \GN{q}^{0*}, \GN{r}^{0*}} \forces_{\PP}
				\gj(\ga_0, \dotsc, \ga_l)$.
\end{proof}
Letting $I$ be $R_\gk$-generic over $V[G][H]$ we get
	the following from the previous corollary together with
		\cref{HodIsInSmallExtension}.
\begin{corollary}
	$\HOD^{V[G][H][I]} \subseteq V[s''G]$.
\end{corollary}
\begin{claim}
	In $V_\gk^{V[G][H][I]}$ all regulars above $\gk_0$
		are
	$\HOD^{V_\gk^{{V[G][H][I]}}}$-measurable.
			\end{claim}	
			\begin{proof}
Since the regulars in the range $[\gk_0, \gk)$ are
	$\setof {\gm_\ga}{\ga < \gk}$,
	we will be done by showing for each $\ga < \gk$ the
	$V$-measure $U_\ga$ lifts to a $\HOD^{V_\gk^{V[G][H][I]}}$-measure.
In $V$, $\gm_\ga$ is measurable.
The set $s''\PE$ is the plain Radin forcing, hence any measure in $V$
	over $\gm_\ga$
	lifts trivially to a measure on $\gm_\ga$
			in $V[s''G]$.
In particular the $V$-measure $U_\ga$
	lifts to the $V[s''G]$
		measure $\Bar{\Bar{U}}_\ga$,
	which is definable by
		$\Bar{\Bar{U}}_\ga = \setof{B \in V[s''G]}
				{\exists A \in U_\ga\ A \subseteq B\subseteq \gm_\ga}$.
				
Since $\HOD^{V_\gk^{V[G][H][I]}} \supseteq V_{(\gm_\ga^{++})_V}$ we get
	$U_\ga \in \HOD^{V_\gk^{V[G][H][I]}} \subseteq \HOD^{V[G][H][I]}
				 \subseteq V[s''G]$.
Let	$\Bar{U}_\ga = \setof{B \in \HOD^{V_\gk^{V[G][H][I]}}}
				{\exists A \in U_\ga\ A \subseteq B\subseteq \gm_\ga}$.
Then	 $\Bar{U}_\ga \in \HOD^{V_\gk^{V[G][H][I]}}$
	and $\Bar{U}_\ga \subseteq \Bar{\Bar{U}}_\ga$.
	Necessarily $\Bar{U}_\ga$	is a measure
	on $\gm_\ga$.
\end{proof}
We get \cref{TheBigPicture} by forcing in $V[G][H][I]$ with
	$\Col(\gw, \lto\gk_0)$.

%
%
%
%
%
%
%
%
%
%
%
%

%
%
%
%\bibliographystyle{plain}
%\bibliography{carmi}
%
\end{document}